\documentclass[12pt]{amsart}

\usepackage{amsmath, amsfonts, amssymb,amsthm}
\usepackage{euscript}
\usepackage[T1]{fontenc}
\usepackage{graphicx}
\usepackage{tikz-cd}
\usepackage{color}

\usepackage{hyperref}

\usepackage{comment}

\newtheorem{thm}{Theorem}[section]
\newtheorem{lem}[thm]{Lemma}

\newtheorem{prop}[thm]{Proposition}
\newtheorem{cor}[thm]{Corollary}

\theoremstyle{definition}
\newtheorem{defn}[thm]{Definition}

\newtheorem{ex}[thm]{Example}
\newtheorem{rem}[thm]{Remark}

\DeclareMathOperator{\R}{\mathbb R}

\DeclareMathOperator{\Pol}{\mathcal P}

\DeclareMathOperator{\NN}{\mathbb N}
\DeclareMathOperator{\QQ}{\mathbb Q}
\DeclareMathOperator{\CC}{\mathbb C}

\DeclareMathOperator{\RSp}{R-Spec}

\DeclareMathOperator{\Sp}{Spec}
\DeclareMathOperator{\Spr}{Spec_r}
\DeclareMathOperator{\ReSp}{R-Spec}

\def \S {{\mathcal S}}

\DeclareMathOperator{\p}{\mathfrak{p}}
\DeclareMathOperator{\q}{{\mathfrak q}}

\DeclareMathOperator{\BAB}{B\otimes_{A} B}

\def \S {{\mathcal S}}

\def \RR {{\mathbb R}}

\def \Na {{\mathcal N}}

\DeclareMathOperator{\GM}{\mathcal{G}}

\title{Real radiciality and monoreal extensions}
\author{Goulwen Fichou, Jean-Philippe Monnier and Ronan Quarez}
\thanks{The authors wish to thank Michel Coste for interesting discussions. They have received support from the project NewMIRAGE n°ANR-23-CE40-0002-01.}


\begin{document}
\maketitle

\begin{abstract} 
We study irreducible polynomials admitting a single real root in any real closed field extension of the base field, called monoreal polynomials. 
We show some stability properties satisfied by the induced monoreal field extensions, and define the monoreal closure of a field. We make the link with a notion of real radiciality for ring extensions, and an injectivity property  at the real spectrum level.
We end with a geometric application, showing that injectivity implies surjectivity for the real spectrum mapping, under certain assumptions.
\end{abstract}

\section*{Introduction}
\subsection*{Motivations}
An algebraic field extension $K\to L$ is said to be purely inseparable  if for every $a\in {\displaystyle L\setminus K}$, the minimal polynomial $m_a$ of $a$ over $K$ is not a separable polynomial (its roots in an algebraic closure are not distinct). This can occur (for a non trivial extension) only when $K$ has positive characteristic $p$. In that case $m_a(x)=x^q-k$, with $k\in K$ and $q$ a power of $p$.
Any algebraic extension $K\to L$ can be uniquely factorized as $K\to L^s\to L$ where $K\to L^s$ is separable and 
$L^s\to L$ is purely inseparable. It appears that, when $K\to L$ is finite, the degree $[L^s:K]$ is equal to the number of $K$-morphisms from $L$ into an algebraic closure of $K$.
As a consequence, the extension $K\to L$ is purely inseparable if and only if any morphism from $K$ into an algebraically closed field $C$ admits exactly one lifting to $L$.

In this paper, we develop an analogue of pure inseparability in the context of real algebra.
Of course, since formally real fields (those fields which admit an ordering) have characteristic zero, such field extensions are separable. Nevertheless, one may study a kind of real version of inseparability, by  considering morphisms into real closed fields. Namely, we call an extension $K\to L$ monoreal if any morphism from $K$ into a real closed field $R$ admits exactly one lifting to $L$.

Purely inseparable extensions are sometimes called radicial extensions. In the context of ring extensions, a radicial extension induces a bijection at the spectrum level, cf. \cite{Gr1}. Hence, one may be interested in a real counterpart with real radicial extensions. We will see that in the real context, injectivity replaces naturally bijectivity.

Our initial interest in developing these notions comes from a geometric motivation: given a real algebraic variety $V$, it is natural to wonder whether there exists a biggest real algebraic variety $W$ between $V$ and its normalization which induces a bijection at the level of real points. For algebraically closed fields, the seminormalization answers this question. However this is hopeless in real algebraic geometry, as proved in \cite{FMQsemi} already for curves, due to the fact that real points may disappear as complex points, preventing surjectivity in general.  Therefore it is much more natural to investigate injectivity in the real framework. This motivation leads us to  study algebraic conditions ensuring that a ring extension $A\to B$ induces an injective map $\Spr B\to \Spr A$ at the real spectrum level (a refinement of the classical spectrum taking into account orderings on residue fields).

\subsection*{Content}
We say that a finite field extension $K\to L$ between formally real fields is monoreal if the minimal polynomial of a primitive element admits exactly one (real) root in any real closed field extension of $K$. Equivalently it means that any morphism $K\to R$ into a real closed field admits exactly one lifting to $L$. In some sense, monoreal polynomials are the opposite of  so-called hyperbolic polynomials (sometimes called also real zeros polynomials).
In the case of number fields, i.e. when $K=\QQ$, the notion of monoreal extension coincides with field extension of signature of type $(1,s)$, which means that the minimal polynomial has $1$ real root and $s$ couples of non real complex conjugate roots. 
We extend our framework of monoreality to the case of algebraic extensions.
In section 1, we show that monoreal extensions are stable under composition and restriction. Then, we show the existence and unicity of monoreal closures 
and we propose a geometric description of the monoreal closure of the field $\R(t)$.
In section 2, we give a characterization of monoreality in term of bijectivity of the induced mapping at the real spectrum level. This characterization is very natural since points in $\Sp_r K$ can be viewed as morphisms from $K$ into real closed fields.

Following our geometric motivation exposed upper, we drop, in section \ref{sec-sper}, the existence condition of the liftings in monoreality, and consider rather extensions $K\to L$ such that any morphism
$K\to R$ into a real closed field $R$ admits at most one lifting. This condition is called real radiciality. It can be characterized by the injectivity of the induced map $\Spr L\to\Spr K$. 
We show that finite real radicial extensions have either odd degrees and are monoreal, or have even degree and $L$ is not formally real.

In section \ref{sec-ring}, we define real radiciality for rings extensions, as a generalization of radiciality introduced in \cite{Gr1}.
We characterized it, in Theorem \ref{prop3}, by the injectivity at the real spectrum level. It is natural to wonder whether this real radiciality condition is related to injectivity for the real part of the classical spectrum, namely  for real prime ideals. If false in general, it happens to be true in the geometric context, cf. Theorem \ref{last-thm}.


\section{Monoreal field extensions}

\subsection{Monoreal extensions}\label{sect-mono}
In all this section, 
$K$ stands for a formally real field.

\begin{defn} 
	We say that an irreducible polynomial $P\in K[X]$ is monoreal if it admits exactly one root in any real closed field extension of $K$.
	
	For an algebraic extension $K\to L$, we say that $a\in L$ is monoreal over $K$ if its minimal polynomial over $K$ is monoreal. The extension $K\to L$ is monoreal if any $a\in L$ is monoreal over $K$.	
\end{defn}	

The trivial extension of a formally real field is monoreal. Note also that the degree of a monoreal polynomial is odd, since if it were even it would admit an even number of real roots (complex roots are conjugate).

We emphasize that, in the case of number fields, the notion of monoreal extension coincides with extension of signature of the type $(1,s)$. Indeed, recall from algebraic number theory that a finite extension $L$ of $\QQ$ is said of signature $(r,s)$ when the minimal polynomial of a primitive element has $r$ real roots and $s$ couples of non real complex conjugate roots. This signature appears in some important results such as Dirichlet's Unit Theorem.

The following elementary criterion of monoreality will be very useful in the paper.

\begin{prop}\label{CriterionFiniteMono}
An algebraic element $a$ is monoreal over $K$ if and only if any morphism $K\to R$ into a real closed field $R$ admits exactly one lifting to $K(a)$. 
\end{prop}	

It relies on the one-to-one correspondence between  liftings of $\phi$ to $K(a)$ and  real roots of  the minimal polynomial $m_a$ of $a$ in $R$.

Since finite extensions are generated by a single element, the following result is natural:

\begin{prop}\label{MonorealCriterion}
The extension $K\to K(a)$ is monoreal if and only if $a$ is monoreal over $K$.
\end{prop}

\begin{proof}
Let us consider a subextension	$K\to K(b) \to K(a)$. Assume that $a$ is monoreal and let us show that $b$ is so.
Consider a morphism $\phi: K\to R$ into a real closed field.
First, there is a lifting of $\phi$ to $K(a)$ and hence, by restriction, to $K(b)$.
Second, let us assume that one has two liftings $\phi_1,\phi_2$ of $\phi$ to $K(b)$. Since the degree of $K\to K(a)$ is odd, this is also the case for the degree of $K(b)\to K(a)$ and hence, $\phi_1,\phi_2$ lift to $K(a)$. Since $a$ is monoreal one gets $\phi_1=\phi_2$, and we are done.
\end{proof}

In the case of algebraic extensions, we no longer have a criterion similar to Proposition \ref{CriterionFiniteMono}~: only one implication remains valid. 
\begin{prop}\label{LiftingAlgebraic}
	If $K\to L$ is monoreal, then any morphism $K\to R$ into a real closed field $R$ admits exactly one lifting to $L$. 
\end{prop}

Remark then that if $K\to L$ is monoreal, then the field $L$ is automatically formally real.

Beware that the converse of Proposition \ref{LiftingAlgebraic} if false. Consider for instance the extension 
$\QQ\to \QQ(\cup_{n\in\NN}2^{\frac{1}{2^n}})$. The unique ordering on $\QQ$ lifts uniquely to an ordering on $\QQ(\cup_{n\in\NN}2^{\frac{1}{2^n}})$, but the extension is not monoreal since $\sqrt{2}$ is not monoreal over $\QQ$.
Another counter-example is given by the extension $\QQ\to \RR_{\rm{alg}}$, where $\RR_{\rm{alg}}$ stands for the fields real algebraic numbers.
The unique ordering of $\QQ$ lifts uniquely to the unique ordering of $\RR_{\rm{alg}}$, but the extension $\QQ\to \RR_{\rm{alg}}$ is not monoreal (since for instance $\sqrt{2}$ is not monoreal over $\QQ$).

\begin{proof}[Proof of Proposition \ref{LiftingAlgebraic}]
	The uniqueness is clear since two different liftings should differ on a certain $K(b)$ where $b\in L$, and then we can use Criterion \ref{CriterionFiniteMono}.	
	For the existence, 
roughly speaking, the idea is to lift to $K(a)$ since $a$ is monoreal, then to take $b\in L\setminus K(a)$ and lift to $K(a)(b)$ and so on until we reach $L$. More precisely, we use Zorn Lemma together with the following Lemma \ref{MonorealDiag} and the fact that an odd degree polynomial always admits a real root in a real closed field.
\end{proof}

\begin{lem}\label{MonorealDiag}
	Let $K\to M$ be an algebraic extension of formally real fields. Consider a subfield $L$ of $M$ and an element $a\in M\setminus L$ which is monoreal over $K$.
		
	Then, the degree of the extension $L\to L(a)$ is odd.
\end{lem}	

\begin{proof}
Let us write $K(a)=K[x]/m_K$ and $L(a)=L[x]/m_L$ where $m_L$ is an irreducible factor of $m_K$ in $L[x]$.

Since $M$ is formally real, so is $L(a)$ and hence, there is a morphism $L(a)\to R$ where $R$ is a real closed field. If $m_L$ had even degree, then $m_L$ would admit at least two real roots in $R$. This is impossible since $m_K$ has a single real root in $R$.
\end{proof}

Another direct application of this lemma concerns the compositum of monoreal extensions:

\begin{prop}\label{MonorealCompositum}
	Let $K\to M$ be an extension of formally real fields. Let us consider two intermediate extensions $L$ and $L'$ which are monoreal over $K$.
	Then their compositum $K(L\cup L')$ in $M$ is monoreal over $K$.
\end{prop}	
\begin{proof}
Let us consider $a\in K(L\cup L')$ and show that $K(a)$ is monoreal over $K$. There exist $b\in L$
and $b'\in L'$ such that $a\in K(b,b')$, so that we are reduced to the case where $L$ and $L'$ are finite over $K$.
	
We claim that any morphism $K\to R$, where $R$  is a real closed field, uniquely extends to $K(b,b')=K(b)(b')=K(c)$, for a certain $c\in M$.
The uniqueness is clear since the image of $b$ and $b'$ are uniquely determined by monoreality of $L$ and $L'$. 
The existence of such a lifting to $K(b)$ comes from the fact that $L$ is monoreal over $K$ and the existence of a lifting to $K(b)(b')=K(c)$ comes from the fact that 
$b'$ has odd degree over $K(b)$ by Lemma \ref{MonorealDiag}.

By Proposition \ref{CriterionFiniteMono} we conclude that $c$ is monoreal over $K$ and hence, by Proposition \ref{MonorealCriterion}, that $a$ is monoreal over $K$.
\end{proof}

Let us note that Proposition \ref{MonorealCompositum} is no longer true without the assumption that $M$ is formally real as showed by the example $K=\QQ$, $M=\CC$, $a=\sqrt[3]{2}$,  $L=\QQ[j\sqrt[3]{2}]$: $L$ and $K(a)$ are monoreal over $K$ but the compositum in $M$ is not.

\begin{rem}
	Let us mention an alternative proof of Proposition \ref{LiftingAlgebraic} without using Zorn Lemma to show the existence of a lifting. 
	Let us define $\phi:L\to R$ by assigning to $a\in L$ the unique root of its minimal polynomial in $R$. To check $\phi$ is a morphism, consider $a,b\in L$ and the corresponding morphisms $\phi_a:K(a)\to R$ and $\phi_b:K(b)\to R$. They factorize by a common morphism $\psi:K(b)(a)\to R$ by Lemma \ref{MonorealDiag}, so that 
	$$\psi(a)+\psi(b)=\phi_a(a) + \phi_b(b)=\phi(a) + \phi(b).$$
	On the other hand $\psi$ restricts to a morphism $K(a+b)\to R$, so by uniqueness in Criterion \ref{CriterionFiniteMono} $\psi(a+b)=\phi(a+b)$. We do similarly for the product.
\end{rem}
\subsection{Compositions and restrictions}

Let us study compositions of monoreal extensions.
The following stability results are quite obvious for finite field extensions by using Criterion  \ref{CriterionFiniteMono}. Our goal is to show that these stabilities remain true in the case of algebraic extensions.
	
\begin{prop}\label{MonorealComposition}
	Let $K\to L\to M$ be a sequence of two monoreal algebraic field extensions.
	Then, the extension $K\to M$ is monoreal.
\end{prop}	
\begin{proof}
	Let $a\in M$ and consider $\phi :K\to R$ with $R$ real closed. Since $K\to L$ is monoreal, $\phi$ lifts to $L$ by Proposition \ref{LiftingAlgebraic}. It remains to use that $L\to M$ is monoreal to get a lifting to $M$ (Proposition \ref{LiftingAlgebraic}) and hence to $K(a)$. Moreover, the image of $a$ by such a lifting is unique by Criterion  \ref{CriterionFiniteMono}, which concludes the proof. 
\end{proof}

Conversely, one has the stability of the monoreality condition under restriction:

\begin{prop}\label{MonorealRestriction}
	Let $K\to L\to M$ be two field extensions.
	If $K\to M$ is monoreal, then the extensions $K\to L$ and
	$L\to M$ are both monoreal.
\end{prop}	
\begin{proof}
Note first that $L$ is formally real since $M$ is.
The extension $K\to L$ is obviously monoreal by definition. 

Let us show now that $L\to M$ is monoreal. 
Let $a\in M$ and $\phi:L\to R$ where $R$ is real closed. Then, $\phi$ induces a morphism $K\to R$. Since $K\to M$ is monoreal, one get a lifting to $K(a)\to R$. Using now Lemma \ref{MonorealDiag}, one may extend $\phi$ to $L(a)$. This lifting is unique since the image of $a$ is uniquely determined by monoreality of $K\to K(a)$. 
\end{proof}

Let us mention that monoreal extensions are also stable under base change.

\begin{prop}\label{MonorealBaseChange}
	Let $K\to L$ be a monoreal field extension embedded in a formally real field which contains some algebraic element $a$. Then, 
	$K(a)\to L(a)$ is monoreal.
\end{prop}	
If $a$ is monoreal over $K$, then it comes from Lemma \ref{MonorealDiag} and the monoreal stability under restriction.
But it is also valid even if $a$ is not monoreal.
\begin{proof}
Let $b\in L(a)$. We want to prove that $b$ is monoreal over $K(a)$. Considering the coefficients of $b$ as a fraction of the variable $a$, one may assume that $L=K(c)$ for some $c$ and $b\in K(c)(a)$ and it is enough to show that $c$ is monoreal over $K(a)$. Namely, we are reduced to prove the proposition in the special case $L=K(b)$ is finite over $K$. 

Using Lemma \ref{MonorealDiag}, the minimal polynomial $m_{b,K(a)}$ of $b$ over $K(a)$ has odd degree. It has then at least one real root in any real closed field extension $R$ of $K(a)$. Since it is a factor of $m_{b,K}$ the minimal polynomial of $b$ over $K$ which has exactly one real root in $R$, one deduces that 
$m_{b,K(a)}$ is monoreal.
\end{proof}

\subsection{Monoreal closure}

One may emphasize the analogy between the notions of monoreal extension and purely inseparable extension.

Indeed, recall that an extension $K\to K(a)$ is purely inseparable if the minimal polynomial of $a$ over $K$ admits exactly one root in any algebraically closed field extension of $K$. Moreover, 
the algebraic extension $K\to L$ appears to be purely inseparable (it means that for any $a\in L$, the extension $K\to K(a)$ is purely inseparable) if and only if any morphism $K\to C$ into an algebraically closed field admits exactly one lifting to $L$.

One may carry on the analogy defining the notion of monoreal closure:

\begin{defn} 
A monoreal closure of $K$ is a monoreal algebraic extension which is maximal (it does not admit any monoreal non trivial algebraic extension).
The monoreal closure of $K$ in an algebraic formally real extension $L$ is the set of all elements in $L$ which are monoreal over $K$.
\end{defn}

\begin{prop} 
The monoreal closure of $K$ in $L$ is a field which is a monoreal extension. It is the maximal monoreal field extension over $K$ contained in $L$.
\end{prop}

\begin{proof}
Proposition \ref{MonorealCompositum} say that, given $a$ and $b$ monoreal over $K$, then
$K(a)(b)$ is monoreal over $K$. Then, $a+b$ and $ab$ are monoreal over $K$ by Proposition \ref{MonorealRestriction}. 	
\end{proof}

Moreover, one has also:
\begin{thm}
The monoreal closure $\overline{K}$ of $K$ exists and it is unique up to a unique $K$-homomorphism.
\end{thm}
\begin{proof}
The existence is given by considering a real closed field $R$ over $K$ and setting $\overline{K}$ to be the 
monoreal closure of $K$ into $R$.
Then, $\overline{K}$ is a maximal monoreal extension since any monoreal extension $L$ of $K$ injects into $R$ by Proposition \ref{LiftingAlgebraic}.

For the uniqueness, let us consider two monoreal closures $\overline{K_1}$	and $\overline{K_2}$ which embed respectively into real closed fields $R_1$ and $R_2$.

According to Proposition \ref{LiftingAlgebraic},
one gets a lifting of $K\to R_1$ to $\phi:\overline{K_2}\to R_1$. By maximality of $\overline{K_1}$ as a monoreal extension, the image $\phi(\overline{K_2})$ is contained into the image of $\overline{K_1}$ in $R_1$, namely $\phi$ induces a morphism $f:\overline{K_2}\to \overline{K_1}$.
Likewise, one has a morphism $g:\overline{K_1}\to \overline{K_2}$. And the compositions $f\circ g$ and $g\circ f$ are identity morphism by Proposition \ref{LiftingAlgebraic}.

Again, Proposition \ref{LiftingAlgebraic} shows the uniqueness of an isomorphism between two monoreal closures.
\end{proof}

	Of course, the monoreal closure of a real closed field $R$ is $R$ itself.
	
	But, in general, it is non obvious to describe elementarily a monoreal closure. For instance, we know that 
	the monoreal closure 
	of $\QQ$ is the subfield of $\RR_{\rm{alg}}$ 
	whose elements are exactly those real algebraic numbers whose minimal polynomial has a single real root. For instance, it contains $\sqrt[r]{n}$ where $n$ is an integer and $r$ is odd. It contains also, for instance, the unique real root of the polynomial $x^3+x+1$. It contains a lot of other roots and seems complicated to describe elementarily.

However, we will be able to describe the monoreal closure of $\RR(t)$ in the forthcoming subsection.

Together with the monoreal closure comes an universal property:
\begin{prop}
	Any monoreal extension of $K$ uniquely embeds into its monoreal closure.
\end{prop}
\begin{proof}
	Let $K\to L$ be monoreal and $\overline{K}$ a monoreal closure of $K$ embedded into a real closed field $R$. 
	By Proposition \ref{LiftingAlgebraic}, on gets a unique morphism $L\to R$ whose image is contained into the image of $\overline{K}$. 
\end{proof}

To end this section, we would like to emphasize that monoreal extensions have a (surprisingly) good behaviour. Indeed, one may try to mimic this section for bireal extensions, namely those extensions for which the minimal polynomial of a primitive element admits exactly $2$ real roots. Unfortunately, those are not stable under composition as shown by the example 
$\QQ\to \QQ(\sqrt{2})\to \QQ(\sqrt{2})(\sqrt{3})$ since the resulting degree 4 extension is totally real (its minimal polynomial admits $4$ real roots). 
Likewise, this notion is not stable under restriction
as shown by the example $\QQ\to \QQ(\sqrt[3]{2})\to \QQ(\sqrt[6]{2})$ since the resulting extension is bireal, but the first one is monoreal.
Analogous examples can be produced for threereal extension and so on.
One may also think about totally real extensions. Although theses are stable under restriction, they are not preserved under composition
as shown by the example 
$\QQ\to \QQ(\sqrt{2})\to \QQ(\sqrt[4]{2})$. 

Besides, one may consider extensions with only an odd number of real roots. This property is stable under composition and restriction since such an extension has odd degree (although it is not stable under compositum).

\subsection{An example : the monoreal closure of $\RR(t)$}
The aim of this section is to describe 
the monoreal closure $C$ of $\RR(t)$ as the set $\GM$ of all continuous maps $f:\RR\to \RR\cup\{\pm\infty\}$ whose graph $G(f)$ coincides with its Zariski closure up to finitely many points.

We begin with describing some elements of $\GM$ to give the flavour of the construction.

\begin{ex}\begin{enumerate}
\item 
The function $f:t\mapsto\sqrt[3]{\frac{1}{t}}$ parametrizes the real root of the monoreal polynomial $P=X^3-\frac{1}{t}\in \R(t)[X]$. Setting $Q=tX^3-1$, we see that the graph of $f$ coincides with the real algebraic set defined by $Q$ in $\R^2$. Note that $f$ goes to the infinity at $0$.

\item The one dimensional part of the cubic given by $P=Q=X^2(X-1)-t^2\in \R[t][X]$, having an isolated point at the origin, admits a parametrization $t\mapsto g(t)$ which is an element in $\GM$. Note that the cubic has two points with coordinates $t=0$, and $g(0)=1$ by continuity.
\end{enumerate}
\end{ex}

Let us recall that orderings on $\RR(t)$ corresponds to cuts $x_+$, $x_-$, where $x$ is a real number, or $-\infty,+\infty$. For instance, non-negative rational functions for the ordering $x_+$ are those functions which remain non-negative on a neighbourhood at the right hand side of $x$.  
Then, any $a\in C$ can be seen as an element of the real closure $\RR(x_+)$ of $\RR(t)$ relatively to the ordering $x_+$. 
Hence, $a$ can be seen as a semi-algebraic continuous function (an half branch of an algebraic curve) defined in a neighbourhood of $x_+$ (cf \cite[8.1.12]{BCR}).
It may happen that this half branch goes to $\pm\infty$.
Likewise, $a$ can be seen as an half branch of an algebraic curve in a neighbourhood of $x_-$ and $\pm \infty$.
It remains now to glue these local data to get a map $f:\RR\to \RR\cup\{\pm\infty\}$.

There is a monic irreducible polynomial $P$ of degree $d$  in $\RR(t)[y]$ such that $P(a)=0$.
Chasing the common denominator $d(t)$, one gets an irreducible polynomial $Q$ in 
$\RR[t][y]$ such that $Q(a)=0$.
Viewing $t$ as a real parameter, 
on $\{ d\not=0\}$, we have a continuous parametrization of the $d$ roots of $Q$ : $\alpha_1(t),\ldots,\alpha_d(t)$.
In fact, the $\alpha_i$ are even analytic if we restrict to $\{\textrm{Disc}(Q)\not=0\}$, where $\textrm{Disc}(Q)$ stands for the discriminant of the polynomial $Q$.
Since $a$ is monoreal, on $\{ d\not=0\}\cap\{\textrm{Disc}(Q)\not=0\}$, there is only one real root, let's say $\alpha_1(t)$ which necessarily coincides with the half branch given by $a$ as defined previously. As a consequence, $a$ is canonically associated to the continuous function $f=\alpha_1(t)$ defined on $\{ d\not=0\}$ and with values in $\RR$. 

To extend $f$ to the whole $\RR$, with possible values in $\RR\cup\{\pm\infty\}$, it suffices to consider the limits of $\alpha_1$. 
Indeed, if $a$ goes to a finite limit at $x_+$, then it  goes to the same finite value at $x_-$ by continuity of the roots, so that $f$ is continuous at $x$. Moreover, if it goes to the infinity at $x_+$, then necessarily $d(x)=0$ and 
$a$ goes also to the infinity at $x_-$. 
This comes from the fact that the induced half-branch at $x_+$ can be parametrized by a Puiseux series of the form $f_+=\sum_{k\geq k_0}a_kt^{\frac{k}{q}}$ where $a_k\in \RR$, $k,k_0,p,q$ are integers and $q$ is odd by monoreality of $a$ (otherwise, there would exists two branches at $x_+$). Now, since $q$ is odd, the unique half-branch induced by $a$ at $x_-$ comes by substituting $t$ with $-t$ in the Puiseux series $f_+$.   

Moreover, the graph of $f$ is closed with respect to Euclidean topology and its Zariski closure captures at most a finite number of points (whose first coordinate $x$ necessarily satisfies $d(x)=0$). Such sets are called Zariski constructible sets, and have been studied for instance in \cite{P}.
So we define a map $\Phi : C\to \GM$ by setting $\Phi(a)=f$. Note that the sum and product of elements of $\GM$ remain in $\GM$ ; it comes from the fact that the image of Zariski closed constructible sets by injective regular maps remains Zariski closed, cf. \cite{P}. As a consequence $\GM$ may be equipped with a field structure.

\begin{prop}
	The map $\Phi$ is a field isomorphism.
\end{prop}	
\begin{proof}
	The injectivity of $\Phi$ being obvious, we focus on  surjectivity.
	Let $f\in\GM$ and consider a polynomial $Q$ in 
	$\RR(t)[y]$ such that $Q(f)=0$. One may assume that $Q$ is irreducible. Indeed, if $Q_1(f)Q_2(f)=0$ with $Q_2(t_0,f(t_0))\not=0$, then 
	$Q_1(t,f(t))=0$ in a neighbourhood of
	$t_0$ and hence $Q_1(t,f(t))=0$ for any $t$.
	
	So one gets a monic irreducible polynomial $P$ in $\RR(t)[x]$ such that $P(f)=0$.
	The function $f$ can be evaluated at any $\gamma$ in $\RR(x_+)$, viewed as an half-branch, which shows the existence of a real root of $P$ in $\RR(x_+)$. 
	Moreover, if there exists another real root of $P$ in $\RR(x_+)$, this will give another half-branch of the graph $G(f)$ of $f$ at $x_+$. This is impossible since $G(f)$ coincides with its Zariski closure except but a finite number of points.
	We argue similarly at $x_-$ and $\pm \infty$, so that we get that $P$ is monoreal, hence $\Phi$ is bijective.

	Finally, it is clear to see that $\Phi$ preserves addition and multiplication, since the definition of the function of $\GM$ associated to an element of $C$ is local, and addition and multiplication in $C$ are transported to the corresponding real closure with respect to the different orderings. 
\end{proof}

\subsection{Polyreal extensions}
Considering that totally real extensions are the counterpart of separable extensions and monoreal  extensions are the counterpart of purely inseparable extensions, one does not have a counterpart of the factorization of an algebraic extension into separable and purely inseparable extensions (Example \ref{Ext3} below will produce a counterexample). 

Although, to get some factorization starting with a monoreal extension, we introduce the notion of polyreal extensions (which contain totally real extensions). 

Again, our framework deals with 
algebraic extensions of formally real fields, since unwanted phenomenons appear when our fields are not formally real.

\begin{defn}
Let $K\to L$ be an algebraic extension of formally real fields. 

We say that $K\to L$ is polyreal if any
$a\in L\setminus K$ is not monoreal, namely if its minimal polynomial admits zero or at least two roots in some real closed field extension of $K$.
\end{defn}

The trivial extension of a field is polyreal.

Here is a criterion which motivates the terminology of polyreal extensions:
\begin{prop}\label{Antimonocriterion}
The algebraic extension $K\to L$ is polyreal 
if and only if, for any subfield $M\to L$ finite over $K$, there is $K\to R$ with $R$ real closed which admits at least two liftings to $M$.
\end{prop}
\begin{proof}
	We know by definition that for any subfield $M\to L$ finite over $K$, there is $K\to R$ with $R$ real closed which admits zero or at least two liftings to $M$.
	
	If it admits zero lifting, then necessarily the degree extension $K\to M$ is even. But, in this case, since $M$ is formally real, one may consider any morphism $M\to R'$ with $R'$ real closed to get an induced morphism $K\to R'$ which admits at least two liftings (again since the degree extension $K\to M$ is even).
\end{proof}	

Recall that if $K\to K(a)$ has even degree, then $a$ is not monoreal over $K$, but it does not necessarily means that 
$K\to K(a)$ is polyreal.

\begin{ex}\label{Ext3}
$\QQ\to \QQ(\sqrt{2})$ is polyreal,
$\QQ\to \QQ(\sqrt[6]{2})$ is not polyreal since 
$(\sqrt[6]{2})^2=\sqrt[3]{2}$ is monoreal over $\QQ$. 

The extension $\RR(x)\to \RR(x)(y)/(y^3-3y+2x)$ is polyreal.
Indeed, its discriminant is $-27\times 4(x^2-1)$, hence it has exactly 3 real roots for any ordering $\alpha$ of $\RR(x)$ satisfying $-1<\alpha<1$.

Any totally real finite extension $K \to L$ is polyreal ; indeed, if it factorizes through a non-trivial monoreal extension $K\to K(a)$, then the maximal number of liftings to $L$ of a morphism from $K$ to a real closed field is greater than the degree of $K(a)\to L$, a contradiction.
\end{ex}

Let us study whether this notion is stable under composition and restriction.

\begin{prop}\label{polyrealComposition}
	Let $K\to L\to M$ be two algebraic polyreal field extensions.
	Then, the extension $K\to M$ is polyreal.
\end{prop}	
\begin{proof}
	Let $a\in M\setminus K$.
	If $a\in L$, then we are done since $K\to L$ is polyreal.
	
	If $a\notin L$, then $L\to L(a)$ is not monoreal by assumption. Using criterion of Proposition \ref{Antimonocriterion}, let us consider $L\to R$ with $R$ real closed which has two different liftings to $L(a)$ (which means that they have distinct images of $a$ in $R$). Then, the induced morphism $K\to R$ admits also two different liftings to $K(a)$ so that $a$ is polyreal over $K$ as expected.
\end{proof}

Concerning restrictions, one has only the obvious following property:

\begin{prop}\label{polyrealRestriction}
	Let $K\to L\to M$ be two field extensions.
	If $K\to M$ is polyreal, then $K\to L$ is polyreal.
\end{prop}

Note that $L\to M$ is not necessarily polyreal by considering the following examples.

\begin{ex}\label{VsPolyCompose}
Let $K=\RR(x)$ and $L_1=K[y]/(y^3+3(x-1)y+2(x+1))$ and
$L_2=K[y]/(y^2-x)$. The discriminant of $L_1$ has the opposite sign than $(x-1)^3+(x+1)=x(x^2-3x+4)$ which is negative for $x>0$ and positive for $x<0$.
As a consequence, any ordering of $K$ associated to an half branch in $\RR_+$ lifts to exactly one ordering in $L_1$ and to two orderings in $L_2$, and any ordering of $K$ associated to an half branch in $\RR_-$ lifts to exactly three orderings in $L_1$ and none in $L_2$. 

Let $M$ be the compositum of $L_1$ and $L_2$ along $K$. Then, $M$ has degree 6 over $K$.
Note that any ordering of $L_2$ lifts to exactly one ordering in $M$, otherwise if it lifts to 3 orderings in $M$, then the induced ordering on $K$ would lift to three different ways to $M$, but it lifts to a single way to $L_1$ and $M$ has degree two over $L_1$, a contradiction.
This shows that $M$ is monoreal over $L_2$.

Moreover, our extension $K\to M$ is polyreal as a composition of polyreal extensions $K\to L_1\to M$. 
\end{ex}
	
\begin{ex}
	Here is a similar example with $K=\QQ(\sqrt{2})$,
	$L_1=K[y]/(y^3+3\sqrt{2}y+2)$ and
	$L_2=K[y]/(y^2-x)$. The sign of the discriminant of $L_1$ is opposite to the sign of $(\sqrt{2})^3+1$, which is negative for the ordering such a that $\sqrt{2}<0$, and positive for
	the ordering such that $\sqrt{2}>0$.
\end{ex}

Note that the same examples prove that polyreal extensions are not stable under base change since $K\to L_1$ is polyreal whereas $L_2\to M$ is monoreal.

Here is another example to show that polyreal extensions are not stable under compositum:

\begin{ex} 	Take as base field $K=\QQ$ and as extensions $L_1=K(a)$ and $L_2=K(b)$, where $a$ and $b$ are the two real roots of the polynomial $P=x^4+x-1$. Then, the compositum in $\RR$
is $\QQ(a,b)$ which has degree 12 over $\QQ$. Note that Ferrari's resolvant of the polynomial $P$ gives an element $s=(a+b)^2$ which has degree 3 over $\QQ$. Indeed, $s$ is the unique real root of the polynomial $x^3+4x-1$. Hence, the extension $\QQ\to \QQ(s)$ is monoreal and the compositum $\QQ(a,b)$ is not polyreal over $\QQ$. \end{ex}

As a consequence, one does not have a good notion of polyreal closure. Nevertheless, one has the following factorization of algebraic morphisms:

\begin{thm}\label{FactorMonoAnti} Any algebraic extension of formally real fields $K\to L$ can be uniquely factorized $K\to \overline{K}\to L$ where $K\to \overline{K}$ is monoreal and $\overline{K}\to L$ is polyreal. \end{thm}
\begin{proof}
Let us define $\overline{K}$ as the monoreal closure of $K$ in $L$. Clearly, $K\to \overline{K}$ is monoreal. Let $a\in L\setminus \overline{K}$. If $a$ were monoreal over $\overline{K}$, then it would be also over $K$, a contradiction. Hence, $\overline{K}\to L$ is polyreal as expected.

Take another such factorization 
$K\to M\to L$. Since $K\to M$ is monoreal, it embeds into $\overline{K}$. Conversely, let $a\in \overline{K}$. Since $a$ is monoreal over $K$, it is also over $M$, but by assumption it is also polyreal over $M$, hence $a\in M$ and we are done.
\end{proof}

Note that we do not have such a result if we reverse the order of monoreal and polyreal extensions. Indeed, let us consider again Example \ref{VsPolyCompose}. One has 
two non-isomorphic decompositions $K\to M\to M$ and $K\to L_2\to M$ where $K\to M$ and $K\to L_2$ are polyreal and $L_2\to M$ is monoreal. Even the existence is not provided ; indeed, take the extension $\QQ\to\QQ(\alpha)\to \QQ(\sqrt{\alpha)}$ where $\alpha$ is the unique (positive) real root of $P=x^3+x-1$.
It does not factorize through a polyreal extension since $x^6+x^2-1$ does not factors through a polynomial of degree $3$ over a real quadratic extension of $\QQ$.

\section{Relation with the real spectrum}\label{sec-sper}

In all this section, $K$ stands for a formally real field.

We aim to characterize which conditions on an algebraic field extension $K\to L$ guaranty the bijectivity of the map $\Spr L \to \Spr K$ induced at the level of the real spectrum.

As already mentioned in the introduction, to require bijectivity between real objects is a quite restrictive condition, and it leads naturally to put rather the stress on injectivity. In the spirit of the monoreality introduced in the first section, we could try  to define the notion of an 
``at most monoreal'' polynomial $P\in K[X]$, saying that it admits at most one real root in any real closed extension of $K$. And likewise an ``at most monoreal'' finite extension of fields $K\to L=K(a)$ would be a finite extension such that the minimal polynomial of $a$ over $K$ is at most monoreal. However, such a notion does not behave well, for instance it is not stable under restriction.
So in order to put the stress on injectivity, we work rather with the uniqueness in the lifting criterion for monoreal finite extension given by Proposition \ref{CriterionFiniteMono}. 

We begin this section with some recalling on the real spectrum, then we discuss the existence and uniqueness in the lifting property, and finally we introduce a real version of the radiciality condition introduced by Grothendieck \cite{Gr1}, which seems to us the relevant condition to discuss injectivity in this context.

\subsection{The real spectrum}

The real spectrum $\Sp_r K$ of a field $K$ can be defined as the set of all its (total) orderings. A point of $\Sp_r K$ can be though equivalently as :
\begin{enumerate}
\item an ordering compatible with the field operations, 
\item a positive cone, which is a subset $P$ of $K$ not containing -1, containing all squares, stable under addition, multiplication, and such that $K=P\cup -P$,  
\item a morphism from $K$ into a real closed field.
\end{enumerate}

For an element $a\in K$, we denote by $\{a>0\}$ the set of morphisms $\alpha:K\to R$, with $R$ real closed, such that $\alpha(a)>0$ for the unique ordering in $R$. The real spectrum is endowed with a topology, with a basis of open neighbourhoods given by the sets $\{ a_1>0,\ldots , a_r>0\}$.

One may also define the real spectrum of a ring as couples $(\p,\leq)$ where $\p$ is a prime ideal of $A$ and $\leq$ is an ordering the residual field 
$k(\p)$. Equivalently, a point $\alpha$ in $\Sp_rA$ can be viewed as a morphism from $A$ to a real closed field $R$. Again, one has a natural topology on $\Sp_rA$. For more details, we refer to \cite{BCR}.\\

One point we would like to emphasize concerns amalgamation of real closed fields in relation with the definition of the real spectrum. 
In fact, to be precise, a point of the real spectrum is an equivalence class of morphisms into  real closed fields. The equivalence relation is the relation generated by the following condition: two morphisms $\alpha:A\to R_1$ and $\beta:A\to R_2$ give rise to the same point in $\Sp_rA$ if there is a morphism $\phi: R_1\to R_2$ such that $\beta=\phi\circ\alpha$. In practice it is more convenient to compare morphisms with a common target field. Actually, 
if  $\alpha_1:A\to R_1$ and $\alpha_2:A\to R_2$ give rise to the same point in $\Sp_rA$, then there are two morphisms 
$\beta_1: R_1\to S$ and $\beta_2: R_2\to S$
where $S$ is a real closed field and such that $\beta_1\circ\alpha_1=\beta_2\circ\alpha_2$. This fact is a consequence of the amalgamation property for real closed fields, which can be stated as follow  (for a proof, see \cite[Ex 1.4.3 P35]{Sch}). 

\begin{prop}\label{Amalgamation}
	Let $\alpha_1: R\to R_1$ and $\alpha_2: R\to R_2$ two real closed fields extensions of a real closed field $R$. Then, there is a common real closed  extension $S$ of $R_1$ and $R_2$, namely two morphisms 
	$\beta_1: R_1\to S$ and $\beta_2: R_2\to S$, where $S$ is a real closed field and such that $\beta_1\circ\alpha_1=\beta_2\circ\alpha_2$.
\end{prop}

In the situation above, it suffices to use the amalgamation property
with the real closure $R$ of the common ordering on the residue field at the kernel of the morphisms $\alpha_1$ and $\alpha_2$.

\subsection{The lifting property}

We say that an extension $K\to L$ satisfies the lifting property if any morphism $K\to R$ into a real closed field $R$ admits exactly one lifting to $L$. 
It is an instance of the so-called substitution property, which is an important notion in real algebraic geometry. One says that the ring extension $A\to B$ satisfies the substitution property if any morphism $A\to R$, where $R$ is a real closed field, admits exactly one lifting to $B$ \cite{FMQ-subs}.
One sees by Proposition \ref{LiftingAlgebraic} that when $A$ and $B$ are formally real fields and the extension is algebraic, the  monoreality property implies the substitution one.
In fact, this substitution property is valid when $A=\Pol(V)$ is the ring of polynomials functions on a non-singular real algebraic set $V$ and $B=\Na (V)$ is the ring of Nash functions on $V$ (\cite[Th 8.5.2]{BCR}). 
Roughly speaking, it means that Nash functions can be uniquely evaluated at any point of $V$.
The substitution property is also valid when $A=\Pol(V)$ is the ring of polynomials on a real algebraic set $V$ and $B=\S(V)$ is the ring of semi-algebraic functions on $V$. Beware that this property is not stable by restriction since it does not remain valid for some subrings of $\S(V)$. For others geometric substitution properties, we refer to \cite{FMQ-subs}.\\

In our setting of an algebraic extension, we will study apart the existence and the uniqueness in the lifting property.

In fact, the uniqueness in the lifting property corresponds to a real version of the radiciality introduced by Grothendieck \cite{Gr1} (see also \cite{BFMQ}). So we rephrase the uniqueness in the lifting property as a real radiciality as follows:

\begin{defn}
Let $K\to L$ be a field extension. We say that $K\to L$ is real radicial if, for any real closed field $R$, and any field homomorphisms $K\stackrel{i}{\to} L\stackrel{\phi_1}{\to}R$
and $K\stackrel{i}{\to} L\stackrel{\phi_2}{\to}R$
such that $\phi_1\circ i=\phi_2\circ i$, then $\phi_1=\phi_2$.	
\end{defn}

Beware that here $L$ is not necessarily formally real. In fact, any $K\to L$ is real radicial as soon as $L$ is not formally real. This shows  that the condition of real radiciality is not stable under restriction, as illustrated by the extensions $\QQ\to \QQ(\sqrt{2})
\to \QQ(i,\sqrt{2})$. \\

 We begin with describing the relationship between these notions and conditions on the map induced at the real spectrum level. 

\begin{prop}\label{LiftSpecrBij} Let $K\to L$ be an algebraic extension. 
\begin{enumerate}
\item Any morphism $K\to R$ into a real closed field $R$ admits a lifting to $L$  if and only if 	the induced map on the real spectrum is surjective.
\item The extension $K\to L$ is real radicial 
 if and only if 	the induced map on the real spectrum is injective.
\end{enumerate} 
\end{prop}

\begin{proof}\begin{enumerate}
\item 
Assume first the existence in the lifting property holds.
A point of the real spectrum of $K$ is given by a morphism $\phi : K\to R$ into a real closed field $R$. This morphism $\phi$ lifts to $L$ which proves the surjectivity on the real spectrum.

Conversely, let $\phi:K\to R$ be a morphism into a real closed field $R$, inducing an ordering $\alpha$ on $K$. 
By surjectivity on the real spectrum, there is a morphism $L\to S$ which induces  $\alpha$, and we may assume that $S$ is a real closed field extension of $R$ by amalgamation.
Since $K\to L$ is algebraic, one may factor the morphism $L\to S$ through  the real closure $k(\alpha)$ of $\alpha$ (see \cite{BCR} for a definition of the real closure), by the universal property of the real closure. We obtain this way a lifting $L\to k(\alpha) \to R$ as required. 

\item Assume that the uniqueness in the lifting property holds. Let us consider two morphisms 
$\phi_1:L\to R_1$ and $\phi_2:L\to R_2$ which coincides as orderings when restricted to $K$.
By amalgamation, one may assume that $R_1=R_2$. By unicity of the lifting, one gets the injectivity on the real spectrum.

Conversely, let $\phi_1:L\to R$ and  $\phi_2:L\to R$ be morphisms in a real closed field $R$ with $\phi_1\circ i=\phi_2 \circ i$. Then  $\phi_1$ and $\phi_2$ represent the same element of $\Sp_r L$ by injectivity at the real spectrum level, and we aim to prove that they are equal. Let $a\in L$, and consider $\phi_1$ and $\phi_2$ on restriction to the finite extension $K(a)$ of $K$. 
By Proposition \cite[Prop 1.3.7]{BCR}, 
$\phi_1$ and $\phi_2$ correspond to the same choice of a root of the minimal polynomial of $a$ in $R$, so $\phi_1=\phi_2$ as expected.
\end{enumerate}
\end{proof}

The following result clarifies the relationship between the monoreality condition introduced in subsection \ref{sect-mono} and bijectivity at the real spectrum level.

\begin{thm}\label{MonoSpecrBij}
	A finite extension $K\to L$ is monoreal if and only if 
	the induced map on the real spectrum is bijective.
	If an algebraic extension $K\to L$ is monoreal, then its induced map on the real spectrum is bijective.
\end{thm}

\begin{proof}
The first point comes from the lifting criterion for finite extension in Proposition \ref{CriterionFiniteMono} together with Proposition \ref{LiftSpecrBij}. For algebraic extension, monoreality implies the lifting criterion by Proposition \ref{LiftingAlgebraic}, hence the result by Proposition \ref{LiftSpecrBij}.
\end{proof}

As already seen, the converse implication of the second statement does not hold. Indeed, consider the extension $\QQ\to \RR_{\rm{alg}}$; it induces a bijection on the real spectrum, but is not monoreal.

Note that the bijection of Theorem \ref{MonoSpecrBij} is in fact a homeomorphism on the real spectrum.

\begin{prop}\label{BijIsHomeo}
An algebraic real radicial extension induces a homeomorphism onto its image in the real spectrum. In particular, an algebraic monoreal extension induces a homeomorphism.
\end{prop}

\begin{proof} The proof of the two cases are similar, so let us deal with the case of a monoreal extension $K\to L$. Denote by $\phi$ the induced induced application onto the real spectrum
$\Sp_r L\to \Sp_r K$ .
Then, $\phi$ is continuous and bijective. Moreover, the real spectrum of any field is compact and a totally disconnected Haussdorf space (\cite[Chap. 3,§2, Lemma 2.8]{HM}, and hence $\phi$ is an homeomorphism.
\end{proof}

\begin{rem}
Let $K\to L$ be monoreal. One may describe in a constructive way the image of an open subset by the induced map on the real spectrum $\Sp_r L\to \Sp_r K$. Let $U=\{b_1>0,\ldots,b_n>0\}$ be a basic open subset in $\Sp_r L$ with $b_1,\ldots,b_n$ in $L$.
Its image is the set of restrictions to $K$ of morphisms $\beta:L\to R$ satisfying $\beta(b_1)>0,\ldots,\beta(b_n)>0$, where $R$ is a real closed field.

Consider the minimal polynomial $m_i$ of $b_i$ over $K$. Forgetting the subscript $i$ for simplicity, write it in the form $x^{n_0}+a_{n_1}x^{n_1}+\ldots a_{n_{r}}x^{n_{r}}+
a_{n_{r+1}}$, where all $a_k$'s are non zero.

Since $K\to K(b_i)$ is monoreal, each of these polynomials $m_i$ admits exactly one real root in $R$.

Descartes's rule of signs says that the number of real roots of $m_i$ and the sign variations in the sequence $(a_{n_1},\ldots,a_{n_r})$ have the same parity. 
Hence, the condition that the unique real root of $m_i$ in $R$ is non negative is given by a disjunction of conjunction of sign variations conditions, namely of the form $a_{n_{k}}a_{n_{k+1}}<0$. This condition defines an open subset of $\Sp_rK$ which we denote by $V$. 

We may conclude that the image of 
$U$ is contained in $V$. The reverse inclusion comes from the surjectivity on the real spectrum since $\phi$ is monoreal.	
\end{rem}

Finite real radicial extensions lead to the following alternative: 
 \begin{prop}\label{FiniteAtMost}
Let $K\to L$ be a finite real radicial extension. Then, either it has odd degree and it is monoreal, either 
it has even degree and it is not monoreal and in this latter case, $L$ is not formally real.  	
 \end{prop}
 
In other words, for finite real radicial extensions either there is always one lifting to $L$ for any $K\to R$ with $R$ real closed, or there is no lifting to $L$ for any $K\to R$ with $R$ real closed.

\begin{proof}[Proof of Proposition \ref{FiniteAtMost}]
If it has odd degree, then any ordering on $K$ lifts to $L$, and since $K\to L$ is real radicial, it lifts uniquely. By Theorem \ref{MonoSpecrBij}, $K\to L$ is monoreal.

If it has even degree, then any ordering on $K$ does not lift to $L$, otherwise there would exists at least two liftings, contradicting that $K\to L$ is real radicial. Then, there is no ordering on $L$ compatible with those of $K$, and hence there is no ordering on $L$, namely $L$ is not formally real.
\end{proof}	

For an algebraic extension, the situation may be more involved. Indeed, we may have a real radicial extension $K\to L$  with $L$ formally real, admitting only subextensions of even degree. And moreover, depending on the choice of an ordering on $K$, there may be either zero or one lifting to $L$.

\begin{ex}\label{ex-root}
Consider 
$$K=\RR(X)\to \RR(X^{1/2})\to \RR(X^{1/4})\to\ldots\to \RR(X^{1/2^n})\to\ldots \RR(\cup_nX^{1/2^n})=L.$$ There is no liftings to $L$ of the ordering $O_-$ (i.e. the ordering for which the non-negative polynomials are those polynomials non-negative on a neighbourhood of $O_-$) and exactly one lifting to $L$ of the ordering $O_+$ (the non-negative polynomials are those polynomials which are non-negative on a neighbourhood of $O_+$). 
\end{ex}

Finally, using Proposition \ref{FiniteAtMost}, we obtain a characterization of monoreality inside real radicial algebraic extensions. 

\begin{cor}
A real radicial algebraic extension is monoreal if and only if any finite subextension has odd degree.
\end{cor}

\section{Real  radiciality for rings extension}\label{sec-ring}

We are concerned now with developing a real version of radiciality for ring extensions. The classical notion of radiciality guaranties that the induced morphism at the spectrum level is injective (see \cite{BFMQ}). The real counterpart will concern the real spectrum. 

Recall that the real spectrum of a ring $A$ can be seen as the set of couples $(\p,\leq)$ where $\p$ is a prime ideal of $A$ and $\leq$ is an ordering the residual field 
$k(\p)$. 
For $\alpha=(\p,\leq)\in\Sp_r A$, the induced prime ideal $\p$ is real, namely the quotient ring $A/\p$ admits an ordering. The set of all such prime ideals will be denoted as $\RSp A$; it is a subset of $\Sp A$. Hence, the so-called support map, which sends $\alpha$ to $\p$, goes from $\Sp_r A$ to $\Sp A$, with image $\RSp A$.

As a typical example, one may think as the real spectrum of $\RR[x]$ as the set of all points $a\in\RR$ which correspond to the evaluation of $x$ at the point $a$, together with the set of all half branches $\gamma(t)$ on $\RR$ which correspond to the evaluation of $x$ at the real puiseux series $\gamma(t)$ (these latter correspond to the orderings of the fraction field $\RR(x)$).

\subsection{Real radiciality}
We extend our notion of real radicial extension to rings:
\begin{defn}
Let $A\xrightarrow{i}  B$ be a rings extension. The extension is real radicial if for any real closed field $R$, if $\psi_1:B\to R$ and $\psi_2:B\to R$ are two distinct morphisms, then they remain distinct by composition with $i$.
\end{defn}

Real radiciality can be characterized in terms of tensor product as follows.

\begin{prop}\label{rrsaturdeux}
Let $A\xrightarrow{i}  B$ be an extension of rings. Then, the following properties are equivalent:
\begin{enumerate}
\item $i$ is  real radicial.
\item The map $\Sp_r (\Delta):\Sp_r B\to \Sp_r (\BAB)$ is surjective where $\Delta:\BAB\to B$ maps $b_1\otimes b_2$ to $b_1b_2$ is the canonical diagonal morphism.
\end{enumerate}
\end{prop}
\begin{proof}
\textbf{(2) $\implies$ (1):}
Assume (2). Let $\psi_1:B\to R$ and $\psi_2:B\to R$ be two morphisms 
such that $\psi_1\circ  i=\psi_2\circ i$. We get a morphism $\psi:\BAB\to R$ such that $\psi(b_1\otimes_A b_2)=\psi_1(b_1)\psi_2(b_2)$. By (2), the point in $\Sp_r(\BAB)$ given by $\psi$ is the image of an element of $\Sp_r B$ by $\Sp_r(\Delta)$. So, up to replacing $R$ by a real closed extension, we get 
a morphism $\psi^{\prime}:B\to R$ such that the following diagram
$$\begin{tikzcd}
\BAB \arrow[d, "\Delta" '] \arrow[r,"\psi"] & R\\
B \arrow[ur,,"\psi'" ',] &
\end{tikzcd}$$
is commutative. Let $b\in B$. We have 
$$\psi_1(b)=\psi(b\otimes_A 1)=\psi'\circ \Delta(b\otimes_A 1)=\psi'(b)=\psi'\circ \Delta(1 \otimes_A b)=\psi(1\otimes_A b)=\psi_2(b),$$ so that $\psi_1=\psi_2$.

\textbf{(1) $\implies$ (2):}
Let $\psi:\BAB\to R$ be a morphism. Composing $\psi$ with $b\mapsto b\otimes_{A}1$ and $b\mapsto 1\otimes_{A}b$, we obtain two morphisms $\psi_i:B\to R$, $i=1,2$. Clearly, $\psi_1\circ  i=\psi_2\circ  i$: for $a\in A$ we have $\psi_1\circ i(a)=\psi( i(a)\otimes_A 1)=\psi(1\otimes_A  i(a))=\psi_2\circ  i(a)$. Remark that for $b_1,b_2$ in $B$ we have $\psi(b_1\otimes_A b_2)=\psi_1(b_1)\psi_2(b_2)$. By (1), it follows that $\psi_1=\psi_2$. We get a commutative diagram
$$\begin{tikzcd}
\BAB \arrow[d, "\Delta" '] \arrow[r,"\psi"] & R\\
B \arrow[ur,,"\psi'" ',] &
\end{tikzcd}$$
by setting $\psi'(b)=\psi_1(b)=\psi_2(b)$. We clearly get $(\Sp_r \Delta)(\psi')=\psi$, namely the point in 
$\Sp_r (\BAB)$ given by $\psi$ is the image of the point in 
$\Sp_r B$ given by $\psi'$.
\end{proof}


\subsection{Relation with the real spectrum}
In order to fit into the framework from the previous sections, we will consider extensions of rings whose induced residual field extensions are algebraic. More precisely, since we focus only on real primes, we consider the following:

\begin{defn}
A morphism of rings $f:A\to B$ is called \textit{real residually algebraic} if for any $\p\in \RSp B$ then $k(f^{-1}(\p))\to k(\p)$ is algebraic.
\end{defn}

As an example, integral morphisms are, of course, real residually algebraic. Another example is the class of quasi-finite morphisms which has been introduced by Grothendieck. Recall that a finite type extension $A\to B$ is  quasi-finite if any fibre consists of isolated points or equivalently if, for any $\q$ prime in $B$ lying over some prime $\p$ in $A$, one has that $B_{\q}/{\p}B_{\q}$ is finite over $k(\p)$ (which implies that the residual extension field $k(\p)\to k(\q)$ is finite).

Real radiciality is a good notion to test injectivity at the real spectrum level.

\begin{thm}\label{prop3}
Let $i:A\to B$ be a real residually algebraic extension of rings. Then  $\Sp_r i$ is injective if and only if $i$ is  real radicial.
\end{thm}

\begin{proof} Assume $\Sp_r i$ is injective

Let $\psi_1:B\to R$ and $\psi_2:B\to R$ be morphisms in a real closed field with $\psi_1 \circ i=\psi_2 \circ i$. Thus the morphisms $\psi_1,\psi_2$ induce orderings $\beta_1,\beta_2$ on $B$ lying over the same ordering $\alpha$ on $A$. By assumption, these two orderings are the same, we denote it by $\beta$. Denote by $\q$ the support of $\beta$ and $\p$ the support of $\alpha$. Then $\psi_1$ and $\psi_2$ induces morphisms $\psi_1',\psi_2'$ at the level of residual fields as in the following commutative diagram
$$\begin{tikzcd}
A \arrow[r,"i"] \arrow[d] & B \arrow[d] \arrow[r, "\psi_2"', "\psi_1"] \arrow[r,shift right]& R & \\
k(\p) \arrow[r,"i'"]  & k(\q) \arrow[ur,swap,"\psi_2'", "\psi'_1"'near start] \arrow[ur, shift right] &&
\end{tikzcd}$$
Note that $\Sp_r i'$ is still injective, since we may identify the orderings on a ring with a specified support, with orderings on the corresponding residue field. Moreover $i'$ is an algebraic extension by assumption on $i$, so we conclude that $\psi'_1=\psi'_2$ using the fields case from Proposition \ref{LiftSpecrBij}, and hence $\psi_1=\psi_2$.\\

Assume $i$ is  real radicial.
Let $\beta_1,\beta_2$ be orderings on $B$ lying over the same ordering $\alpha$ on $A$. So there exist morphisms $\psi_1:B\to R_1$ and $\psi_2:B\to R_2$ into real closed fields such that $\psi_1\circ i$ and $\psi_2\circ i$ induce $\alpha$. We can amalgamate $\psi_1\circ i$ and $\psi_2\circ i$ by Proposition \ref{Amalgamation}, giving rise to commutative diagrams:
$$\begin{tikzcd}
& & R_1 &&& & R_1 \arrow[dr,"j_1"]& \\
A \arrow[r,"i"] \arrow[urr, "\psi_1\circ i", bend left=20]\arrow[drr, "\psi_2\circ i", bend right=20] &B  \arrow[dr, "\psi_2"] \arrow[ur, "\psi_1"] &   &&& A \arrow[ur,"\psi_1\circ i"] \arrow[dr,"\psi_2\circ i"] & & R \\
&& R_2 &&& & R_2 \arrow[ur,"j_2"]&
\end{tikzcd}$$
Now $j_1\circ \psi_1:B\to R$ and $j_2\circ \psi_2:B\to R$ are equal after composition with $i$, so by assumption they are equal. It implies the equality of $\beta_1$ and $\beta_2$.
\end{proof}

Let us note that, when the extension $i$ is real radicial, then $\Sp_r i$ is not necessarily an homeomorphism onto its image. Indeed, take the example of the normalization $C'$ of the cubic with one nodal point $C$ where we erase in $C'$ the point $A$ which is one of the two points lying over the nodal point of $C$. Then, an half branch in $C'$ with origin $A$ has as image an half branch in $C$ whose origin is the nodal point and hence is not closed.

\subsection{Relation with real prime ideals}

A natural question is how injectivity at the real spectrum level is related to injectivity at the level of real ideals of the spectrum.

\begin{prop}
\label{prop3bis}
Let $i:A\to B$ be a real residually algebraic extension of rings. Under the two conditions
\begin{enumerate}
\item $\Sp i$ is injective by restriction to $\ReSp B$,
\item for any $\p\in \ReSp B$, $k(i^{-1}(\p))\to k(\p)$ is  real radicial,
\end{enumerate}
the map $\Sp_r i$ is injective.\\
Conversely, the injectivity of $\Sp_r i$ implies (2).
 \end{prop}

\begin{proof}
Assume we have conditions (1) and (2) and consider $\beta_1,\beta_2$ two orderings on $B$ lying over the same ordering $\alpha$ on $A$. As in the proof of Theorem \ref{prop3}, we may assume that $\beta_1$ and $\beta_2$ are given by morphisms in a common real closed field, $\psi_1:B\to R$ and $\psi_2:B\to R$, with $\psi_1 \circ i=\psi_2 \circ i$. By (1) the kernels of $\psi_1$ and $\psi_2$ are equal, so we are reduced as in the proof of Theorem \ref{prop3} to deal with the fields case, which is true by (2).

Conversely, assume $\Sp_r i$ injective. Let $\p\in\ReSp B$. The map $\Sp_{r} i'$,  where $i': k(i^{-1}(\p))\to k(\p)$, is also injective since the data of an ordering on $B$ with support $\p$ is equivalent to the data of an ordering on $k(\p)$. The result follows from the field case by Proposition \ref{LiftSpecrBij}.
\end{proof}

\begin{rem} \label{remprop3} In general, the injectivity of $\Sp_r i$ does not imply (1) in Proposition \ref{prop3bis}. 
Consider for $A$ the ring of polynomial functions on the unit circle localized at the point $(1,0)$, namely $A=(\RR[x,y]/(y^2=1-x^2))_{(1,0)}$. Let now $B$ be the associated ring of germs of semi-algebraic functions 
on the unit circle localized at the point $(1,0)$. Then, $A$ and $B$ have the same real spectrum (\cite[8.5.2]{BCR}), 
but the real ideal $(0)$ in $A$ splits into two real ideals in $B$, namely the class of $(y-\sqrt{1-x^2})$ and $(y+\sqrt{1-x^2})$ in $B$.
\end{rem}

However in the algebraic geometric setting, we will have a better connection. By an affine algebraic variety $X$ over a real closed field $R$, we mean the scheme associated with a finitely generated $R$-algebra. In the following, we always assume the set of real closed points $X(R)$ contains a smooth point. In this situation, the coordinate ring of $X$, denoted by $\Pol(X)$, coincide with the ring of polynomial functions on $X(R)$.

\begin{prop}\label{prop-inj}
Let $\pi : Y\to X$ be a dominant and finite type morphism between affine algebraic varieties over a real closed field, corresponding to a ring extension $\Pol(X)\to \Pol(Y)$. Assume that $\pi$ is injective at the level of real closed points. Then the induced map $\RSp \Pol(Y) \to \RSp \Pol(X)$ is injective.
\end{prop}

\begin{proof}
Let $q_1$ and $q_2$ be real prime ideals of $\Pol(Y)$ lying over a real prime ideal $p\in \Pol(X)$, corresponding respectively to irreducible varieties $W_1$ and $W_2$ of $Y$ and $V$ of $X$. Note that $\dim W_i \geq \dim V$ since $\pi(W_i)$ is Zariski dense in $V$ and $\pi$ is dominant. Moreover $\dim W_i\leq \dim V$ by injectivity of $\pi$ at the level of real closed points, so that  $W_1, W_2$ and $V$ have the same dimension $d$.

By injectivity of $\pi$ at the level of closed points, $\pi(W_1)$ is a Zariski-constructible subset of $V$ of dimension $d$ by an argument similar to \cite{P} (while \cite{P} deals with the case $R=\R$, one may replace Theorem 2.2 there with Corollary 3.19 in \cite{Cos} which, even if it is written in the real case, is valid with the same proof over a real closed field). 
The same is true for $\pi(W_2)$, so the irreducibility of $V$ imposes that the intersection $\pi(W_1) \cap \pi(W_2)$ is still a Zariski-constructible subset of $V$ of dimension $d$. The injectivity of $\pi$ at the level of closed points implies that $W_1$ and $W_2$ must intersects in dimension $d$ too, so that they must be equal by irreducibility. As a consequence and since they are real ideals then $q_1=q_2$ as expected.
\end{proof}

This result implies a good understanding of the connection between injectivity at the real spectrum level and injectivity at the level of real ideals of the spectrum in the geometric context.

\begin{thm}\label{last-thm} Let $A$ and $B$ be finitely generated algebras over a real closed field, and $i:A\to B$ be a real residually algebraic extension of rings. Then,  $\Sp_r i$ is injective if and only if :
 $\Sp i$ is injective by restriction to real prime ideals and
$k(i^{-1}(\p))\to k(\p)$ is  real radicial for any $\p\in \ReSp B$.
\end{thm}

\begin{proof}
According to Proposition \ref{prop3bis}, it remains to prove that the injectivity of $\Sp_r i$ implies the injectivity of $\Sp i$ by restriction to $\ReSp B$. But the injectivity of $\Sp_r i$ implies the injectivity of $\Sp i$ at the level of real closed points since a real closed field admits a unique ordering, so it suffices to apply Proposition \ref{prop-inj}.
\end{proof}

Here is an application in the framework of integral extensions together with an assumption of centrality.
Recall that integral extensions are real residually algebraic, and moreover provide us a real going up property for orderings (see \cite[Proposition 4.3]{ABR}). On the other hand, a central domain $A$ is such that the real spectrum of its fraction field is dense in $\Sp_r A$ (see \cite{FMQ17} for more details). 

\begin{prop}
Let $\phi:A\to B$ be a finite extension between real domains whose induced mapping onto the real spectrum is injective. Assume that $A$ is central. Then, $\phi$ induces a bijection onto $\Sp_r A$.
\end{prop}	

\begin{proof}
The extension between the fraction fields $K\to L$ is finite and induces an injection onto the real spectrum. 
Since our fields are formally real, by Proposition \ref{FiniteAtMost}, this extension is monoreal and induces a bijection onto the real spectrum. 

Then, the image of $\Sp_r B$ contains $\Sp_rK$. Since $A$ is central and $\Sp_r \phi$ is closed by the real going up for orderings, we are done.
\end{proof}	

Note that the result is no longer true without the finite assumption, even for algebraic field extensions, cf. Example \ref{ex-root}.


\begin{thebibliography}{ABR2} 



\bibitem{ABR} C. Andradas, L. Br\"ocker, J.M. Ruiz, {\it Constructible sets in real geometry}, Springer (1996)
    



\bibitem{BFMQ} F. Bernard, G. Fichou, J.-P. Monnier, R. Quarez, {\it Algebraic characterizations of homeomorphisms between algebraic varieties}, Math. Z., 308, No. 1, Paper No. 12, 28 p., (2024)



\bibitem{BCR} J. Bochnak, M. Coste, M.-F. Roy, 
{Real algebraic
  geometry}, Springer, (1998)



\bibitem{Cos} M. Coste, 
{\it Real algebraic sets}, in 
Panor. Synthèses, 24 SMF, Paris, 2007





\bibitem{FMQ-subs} G. Fichou, J.-P. Monnier, R. Quarez, {\it Substitution property for the ring of continuous rational functions}, Singularities-Kagoshima 2017, World Scientific (2020) 71-93 



\bibitem{FMQsemi} G. Fichou, J.-P. Monnier, R. Quarez, {\it     Weak and semi normalization in real algebraic geometry},  Ann. Scuola Norm. Sup. Pisa Cl. Sci. (5) 22, no. 3, 1511--1558, (2021)
 
\bibitem{FMQ17} G. Fichou, J.-P. Monnier, R. Quarez, {\it Hilbert 17th property and central prime cones}, Journal of Algebra 664, 552--594 (2025) 

 


\bibitem{Gr1} A. Grothendieck (rédigé avec la collaboration de J. Dieudonné), {\it \'Eléments de géométrie algébrique I: Le langage des schémas}, Publ. Math. IHES 4, (1960)



\bibitem{HM} J. Milnor, D. Husemoller, Symmetric bilinear forms, Springer-Verlag (1974)









\bibitem{P} A. Parusi\'nski, {\it Topology of injective endomorphisms of real algebraic sets}, Math. Annalen, 328
(2004), 353-372



\bibitem{Sch} C. Scheiderer, \textit{A Course in Real Algebraic Geometry - Positivity and Sums of Squares}, Springer, (2024) 





\end{thebibliography}
\end{document}